\documentclass[12pt,a4paper,reqno,oneside]{amsart}
\usepackage{t1enc}
\usepackage{times}
\usepackage{amssymb}
\usepackage[mathscr]{euscript}
\usepackage{graphicx}
\usepackage{hyperref}

\usepackage{comment}

\usepackage{float}
\usepackage{epstopdf}

\usepackage{tikz}

\numberwithin{equation}{section}

\theoremstyle{plain}
\newtheorem{thm}{Theorem}
\newtheorem{lem}{Lemma}

\theoremstyle{definition}

\newtheorem{expl}{Example}
\theoremstyle{remark}
\newtheorem{remk}{Remark}

\makeatletter
\def\moverlay{\mathpalette\mov@rlay}
\def\mov@rlay#1#2{\leavevmode\vtop{%
   \baselineskip\z@skip \lineskiplimit-\maxdimen
   \ialign{\hfil$\m@th#1##$\hfil\cr#2\crcr}}}
\newcommand{\charfusion}[3][\mathord]{
    #1{\ifx#1\mathop\vphantom{#2}\fi
        \mathpalette\mov@rlay{#2\cr#3}
      }
    \ifx#1\mathop\expandafter\displaylimits\fi}
\makeatother

\newcommand*{\rom}[1]{\expandafter\@slowromancap\romannumeral #1@} 
\renewcommand{\phi}{\varphi} 

\newcommand{\nooutput}[1]{}
\newcommand{\1}{1\!\!\!\!\!\;{\rm I}}
\newcommand{\mbR}{{\mathbb R}}

\newcommand{\cF}{{\mathcal F}}

\newcommand{\sign}{\mathop{\rm sgn}}

\newcommand{\ve}{\varepsilon}

\newcommand{\const}{\mathop{\rm const}}

\newcommand{\be}{\begin{equation}}
\newcommand{\ee}{\end{equation}}

\begin{document}

\title[On a Selection Problem for Small Noise Perturbation in
Multidimensional Case]{On a Selection Problem for Small Noise Perturbation in
Multidimensional Case}

\date{\today}

\author{Andrey Pilipenko}
\address{Institute of Mathematics,  National Academy of Sciences of
Ukraine, Tereshchenkivska str. 3, 01601, Kiev, Ukraine}
\thanks{Research is partially supported by FP7-People-2011-IRSES Project number 295164}

\author{Frank Norbert Proske}
\address{Department of Mathematics, University of Oslo, PO Box 1053 Blindern, N-316 Oslo, Norway}


\begin{abstract}
The problem on identification of a limit of an ordinary differential equation with discontinuous drift that perturbed by a zero-noise is considered in multidimensional case. This  problem is a classical subject of stochastic analysis, see, for example, \cite{BaficaBalsi, Trevisan, Flandoli, KrykunMakhno}. However the multidimensional case was poorly investigated.  We assume that the drift coefficient has a jump discontinuity along a hyperplane and is Lipschitz  continuous in the upper and lower half-spaces. It appears that the behavior of the limit process depends on signs of the normal component of the drift at the upper and lower half-spaces in a neighborhood of the hyperplane, all cases are considered.
\end{abstract}

\maketitle

\section{Introduction}

Consider the Cauchy problem 
\begin{equation}
X_{t}=x+\int_{0}^{t}b(X_{s})ds,t\geq 0  \label{CP}
\end{equation}%
for $x\in \mathbb{R}^{d}$, where $b:\mathbb{R}^{d}\longrightarrow \mathbb{R}%
^{d}$ is a Borel measurable vector field.

If $b$ satisfies a Lipschitz and linear growth condition, it is well-known
that there exists a global unique solution $X_{\cdot }\in C([0,\infty );%
\mathbb{R}^{d})$ to (\ref{CP}).

However, if $b$ is not Lipschitzian, the situation may change dramatically
and well-posedness of (\ref{CP}) in the sense of uniqueness or even
existence of solutions may fail.

An example of such a function is given by%
\begin{equation}
b(x)=2sgn(x)\sqrt{\left\vert x\right\vert }  \label{sgn1}
\end{equation}%
for $X_{0}=0$ and dimension $d=1,$ where the extremal trajectories $%
X_{t}=+t^{2},-t^{2}$ and the zero curve $X_{t}=0,t\geq 0$ are solutions to (%
\ref{CP}) among infinitely other ones.

Another example in the case of a discontinuous vector field is%
\begin{equation}
b(x)=sgn(x)  \label{sgn2}
\end{equation}%
for $X_{0}=0$ with infinitely many solutions, where $X_{t}=+t,-t,t%
\geq 0$ are extremal solutions.

If we merely require that the vector field $b$ is continuous, satisfying a
growth condition of the form $\left\langle b(x),x\right\rangle \leq
K(\left\vert x\right\vert ^{2}+1)$ then it follows from Peano's theorem and
the theorem of Arzel\`{a}-Ascoli that the set $C(x)$ of solutions $X_{\cdot
}\in C([0,\infty );\mathbb{R}^{d})$ of (\ref{CP}) is non-empty and compact
in $C([0,\infty );\mathbb{R}^{d})$. Moreover, $C(x)$ is connected. See \cite%
{St}.

Here the problem of uniqueness of solutions of (\ref{CP}) for an initial
distribution $\mu _{0}$ on $\mathbb{R}^{d}$ , which corresponds to the case,
when $C(x)$ is a singleton for $x$ $\mu _{0}-$a.e., can be characterized by
unique solutions of narrowly mesurable families of probability measures $%
(\mu _{t})_{t\geq 0}$ on $C([0,\infty );\mathbb{R}^{d})$ satisfying the
continuity equation%
\begin{equation}
\mu _{t}(f)=\mu _{0}(f)+\int_{0}^{t}\mu _{s}(b\cdot \triangledown f)ds,t\geq
0  \label{ContEq}
\end{equation}%
for all $f\in C_{c}^{\infty }(\mathbb{R}^{d})$ (space of smooth functions
with compact support). It turns out that such solutions have the
representation $\mu _{t}=\pi _{t}\mu ,t\geq 0$ for projections $\pi _{t}$
and a probability measure $\mu $ on $C([0,\infty );\mathbb{R}^{d})$ called a
superposition of solutions of the ODE (\ref{CP}). See \cite{DL}, \cite{A}
and \cite{F} for more information on the concept of superposition of
solutions.

Using the concept of renormalized solutions, we mention that for $b\in
W^{1,1}(\mathbb{R}^{d})$ with $div(b)=0$ and initial distributions $\mu _{0}$
with $\frac{\partial \mu _{0}}{\partial x}\in L^{\infty }(\mathbb{R}^{d})$
it can be shown that the continuity equation has a unique solution $(\mu
_{t})_{t\geq 0}$ in the subclass of solutions for which $\frac{\partial \mu
_{t}}{\partial x}\in L^{\infty }(\mathbb{R}^{d})$ for all $t\geq 0.$ See 
\cite{DL}, \cite{A} and \cite{AG}.

On the other hand, the case, when $C(x)$ is not a singleton, gives rise to
the natural question of how an "appropriate" or "meaningful" solution to (%
\ref{CP}) can be selected.

One important method in connection with this selection problem is due to
Krylov \cite{Kr}, who constructed Markov selections, that is families of
superposition solutions $(\mu ^{x})_{x\in \mathbb{R}^{d}}$ such that%
\begin{equation*}
\mu _{t+s}^{x}=\int_{\mathbb{R}^{d}}\mu _{s}^{x}\mu _{t}^{x}(dy),t,s\geq
0,x\in \mathbb{R}^{d}
\end{equation*}%
holds for $\mu _{t}^{x}=\pi _{t}\mu ^{x}$.

Another crucial approach which we want to employ in this paper is that of
zero-noise selection, that is the selection of a solution $X_{\cdot }$ to (%
\ref{CP}) as a limiting value of solutions $X_{\cdot }^{\varepsilon }$ of
the ODE (\ref{CP}) perturbed by a small noise $\varepsilon w(\cdot )$ given
by the stochastic differential equation (SDE)%
\begin{equation}
X_{t}^{\varepsilon }=x+\int_{0}^{t}b(X_{s})ds+\varepsilon w(t)
\label{SmallNoise}
\end{equation}%
for $\varepsilon \searrow 0$ in the sense of convergence in law, where $%
w(\cdot )$ is a $d$-dimensional Wiener process. The motivation for this
selection principle comes from the desire to construct solutions to (\ref{CP}%
), which are stable under random perturbations.

The first results in this direction, that is when $C(x)$ is not a singleton
("Peano phenomenon"), was obtained in the foundational papers of Bafico \cite%
{Ba} and Bafico, Baldi \cite{BaficaBalsi} in the case of one-dimensional
time-homogeneous vector fields $b$, where the authors prove under certain
conditions the existence of a unique limiting law on a small time interval
which is concentrated on at most two trajectories. The proof of the latter
results relies on estimates of mean exit times of $X_{\cdot }^{\varepsilon }$
with respect to (small) neighbourhoods of isolated singular points of $b$ by
means of solutions of an associated boundary value problem. We also mention
the papers \cite{GHR}, \cite{H}, where the authors use large deviation
techniques to study the convergence rate of the laws of $X_{\cdot
}^{\varepsilon }$ for a concrete class of one-dimensional time-homogeneous
functions $b$ related to (\ref{sgn1}). In this context it is also worth
mentioning the work of \cite{BS}, which among other things deals with the
study of the small noise problem (\ref{SmallNoise}) based on viscosity
solutions of (perturbed) backward Kolmogorov equations in the scalar case.
See also the Malliavin calculus approach in \cite{MMP} and the article \cite%
{Trevisan} based on local time techniques.

We also remark that extensions of the paper \cite{BaficaBalsi} to the case of
zero-noise limits of linear transport equations associated with the
one-dimensional ODE%
\begin{equation*}
dX_{t}=2sgn(X_{t})\left\vert X_{t}\right\vert ^{\gamma },\gamma \in (0,1)
\end{equation*}%
were analyzed in \cite{AF}, \cite{Att}. See also \cite{DLN}, \cite{Ma} in
the case of zero-noise limits of non-linear PDE's.

\bigskip

Let us now have a look at the small noise problem (\ref{SmallNoise}) in the
multidimensional case. In fact the muldimensional problem is scarcely
treated in the literature. Here we shall distinguish between the continuous
and discontinuous vector fields:

In the case of bounded and continuous functions $b$ Zhang \cite{Zh} gives a
characterization of the limiting values $X_{\cdot }^{\varepsilon }$ of (\ref%
{SmallNoise}) by using viscosity solutions of Hamilton-Jacobi-Bellman
equations in connection with so-called exit time functions, which partially
extends results in \cite{BaficaBalsi} to the multidimensional setting.

To the best of our knowledge, the case of discontinuous multidimensional
vector fields $b$ has been only examined in the papers of Delarue, Flandoli,
Vincenzi \cite{DFV} and \cite{BOQ}. In the remarkable work \cite{DFV} the
authors study small noise perturbations of the Vlasov-Poisson equation by
means of estimates of probabilities for exit times in connection with a
zero-noise limit for ODE's in four dimensions. The paper \cite{BOQ} deals
with ODE's (\ref{CP}) for merely measurable $b$. However, the concept of
solutions to (\ref{CP}) in the latter work is in the sense of Filippov,
which we don't want to consider in this article.

The objective of our paper is the analysis of zero-noise limits in the case
of discontinuous time-inhomogeneous vector fields $b$ in $\mathbb{R}^{d}$.
More precisely, we aim at considering vector fields $b$, whose discontinuity
points are located in a hyperplane. Our method for the construction of
zero-noise limits, which is different from the techniques of the above
mentioned authors, is based on estimates of probabilities for exit times of $%
X_{\cdot }^{\varepsilon }$ at discontinuity points in a hyperplane. We
comment that our approach extends the one in \cite{BaficaBalsi} to the
multidimensional case. In contrast to \cite{BaficaBalsi} our technique does not
require knowledge of the explicit distribution of $X_{\cdot }^{\varepsilon }$%
. We in fact show that the behavior of the limiting process depends on the
normal component of the drift at the upper and lower half-spaces in a
neighbourhood of the hyperplane.

\section{Main results}

Consider an SDE in $\mbR^d$ with a small noise parameter
\be\label{eq1.1}
X^\ve(t)=x^0+\int_0^tb(s,X^\ve(s))ds+\ve w(t),
\ee
where $w(t), t\geq 0,$ is a Wiener process.

Assume that
$$
b(t,x)=\begin{cases}
b^+(t,x), \ x_d\geq 0\\
b^-(t,x), \ x_d< 0
\end{cases}=b^+(t,x)\1_{x_d\geq 0}+b^-(t,x)\1_{x_d< 0},
$$
where $x=(x_1,\dots,x_d)=(\bar x,x_d),$ $b^\pm$ are measurable locally bounded functions that satisfy a Lipschitz condition in $x$ on
$\mbR^d$.

There exists a unique strong solution to SDE \eqref{eq1.1} by Veretennikov's theorem \cite{Veretennikov}.

Note that $X^\ve$ spends zero time at the hyper-plane $H:=\{x\in\mbR^d\  : \ x_d=0\},$ so it does not
matter how to define the drift coefficient if $x_d=0.$

Consider a formal limit equation for \eqref{eq1.1}
\be\label{eq2.1}
X(t)=x^0+\int_0^tb(s,X(s))ds.
\ee
Since $b$ is Lipschitz continuous in $x$ outside of $H,$   equation \eqref{eq1.1} has a unique solution up to the moment $\tau$ of hitting $H$,
$$
\tau_H:=\inf\{t\geq 0: \ X(t)\in H\}.
$$
Set
$$
\tau^{(\ve)}_H:=\inf\{t\geq 0: \ X^\ve(t)\in H\}.
$$
\begin{thm}\label{thm1}
Let $x^0\notin H.$ Then
we have   the following convergence   with probability 1
$$
X^\ve (\cdot\wedge \tau_H)\to X (\cdot\wedge \tau_H), \ \ve\to 0,
$$
where $X^\ve, X$ are considered as random elements with values in  the space of continuous functions $C([0,\infty), \mbR^d)$
with the topology of uniform convergence on compact sets.

Moreover, if $\mp b^\pm(X(\tau_H))>0,$ then
$
 \tau^{(\ve)}_H \to \tau_H, \ \ve\to 0 \ \ \ \mbox{a.s.}
$
\end{thm}
The proof is standard.
\begin{remk}
If $x^0\notin H$ and
$$
\exists  c>0\ \forall t\geq 0\ \forall x\in\mbR^d_+:\ \sign(x_d)b_d(t,x)\geq -c x_d,
$$
 then $\tau=+\infty $ and  therefore $X^\ve \to X, \ve\to 0 $ a.s. in $C([0,\infty), \mbR^d)$.
\end{remk}

For any initial condition $x_0$ the following result is fulfilled.
\begin{lem}\label{lem_rel_comp}
The sequence of distributions of $\{X^\ve, \ve\in (0,1)\}$ is weakly relatively compact in $C([0,\infty); \mbR^d).$

For any limit point $X$   of $\{X^\ve \}$  as $\ve\to 0+$ and for any $t_0\geq 0$ the following equality holds a.s.
$$
X(t)=X(t_0)+\int_{t_0}^t
b(s, X(s))ds, t\in [t_0, \tau_{t_0,H}],
$$
where $\tau_{t_0,H}=\inf\{t\geq t_0: X(t)\in H\}.$

\end{lem}
The proof of the Lemma is standard.
\begin{remk}\label{rem_subseq}
To prove that the sequence  $\{X^\ve, \ve\in (0,1)\}$ converges in distribution to a process $X^0$ as $\ve\to 0+$, it is sufficient to show that for any sequence $\{\ve_k\}$,
  $\lim_{k\to 0} \ve_k=0$ there exists a subsequence $\{\ve_{k_l}\}$ such that $ X^{\ve_{k_l}}\Rightarrow X^0, \ l\to\infty.$ Since the family $\{X^\ve, \ve\in (0,1)\}$ is weakly relatively compact, without loss of generality we may initially assume that $\{X^{\ve_k}\}$ is already convergent.
\end{remk}

To describe the behavior of    processes after the hitting $H,$ let us assume that the initial starting point $x_0$ belongs to $H$.

Consider the following cases.

{\bf  A1} $\exists T>0\ \exists \delta>0\ \exists c>0\ \forall t\in[0,T]\ \forall x\in\mbR^d\setminus H, |x-x^0|\leq \delta:\ \
\sign(x_d) b_d(t,x)\geq c;$


{\bf  A$2_+$} $\exists T>0\ \exists \delta>0\ \exists c>0\ \forall t\in[0,T]\ \forall x\in\mbR^d_+, |x-x^0|\leq \delta:\ \
b_d^+(t,x)\geq c$ and $\forall x\in H, |x-x^0|\leq \delta:\ \
b_d^-(t,x)\geq 0;$


{\bf  A$2_-$} $\exists T>0\ \exists \delta>0\ \exists c>0\ \forall t\in[0,T]\  \forall x\in\mbR^d_-, |x-x^0|\leq \delta:\ \
b_d^-(t,x)\leq -c$ and $\forall x\in H, |x-x^0|\leq \delta:\ \
b_d^+(t,x)\leq 0;$

{\bf  A3} $\exists T>0\ \exists \delta>0\ \exists c>0 \forall t\in[0,T]\ \forall x\in\mbR^d\setminus H, |x-x^0|\leq \delta:\ \
\sign(x_d) b_d(t,x)\leq -c;$


{\bf  A3$_+$} $\exists T>0\ \exists \delta>0\ \exists c>0 \forall t\in[0,T]\ \forall x\in\mbR^d_+, |x-x^0|\leq \delta:\ \
b_d^+(t,x)\leq -c$ and $\forall x\in H, |x-x^0|\leq \delta:\ \
  b_d^-(t,x)=0;$

{\bf  A3$_-$} $\exists T>0\ \exists \delta>0\ \exists c>0 \forall t\in[0,T]\ \forall x\in\mbR^d_+, |x-x^0|\leq \delta:\ \
b_d^-(t,x)\geq c$ and $\forall x\in H, |x-x^0|\leq \delta:\ \
  b_d^+(t,x)=0;$

{\bf A4}  $\exists T>0\ \exists \delta>0\  \forall t\in[0,T]\ \forall x\in H, |x-x^0|\leq \delta:\ \
  b_d^\pm(t,x)=0.$

Assume that A1 holds. Then there exist unique (local) solutions to \eqref{eq2.1} that
leave $H$ to the positive half-space or negative half-space, respectively. Denote them by $X^+(t)$ and $X^-(t)$,
\be\label{eq4.1}
X^\pm(t)=x^0+\int_0^tb^\pm(s, X^\pm(s))ds, \ t\in [0,\tau^\pm_H],
\ee
where
\be\label{eq4.2}
\tau^\pm_H=\inf\{t>0 \ : \ X^\pm(t)\in H\}.
\ee

\begin{thm}\label{thm3}
Assume that A1 is satisfied.
The distribution of $X^\ve(\cdot\wedge \tau^+_H\wedge\tau^-_H)$ converges weakly as $\ve\to 0$ to the  measure
$$
p_-\delta_{X^-(\cdot\wedge \tau^+_H\wedge\tau^-_H) }+p_+\delta_{X^+(\cdot\wedge \tau^+_H\wedge\tau^-_H) },
$$
where
$$
p_\pm=\frac{\pm b_d^\pm(0,x^0)}{b^+_d(0,x^0)-b^-_d(0,x^0)}.
$$
\end{thm}
\begin{remk}
Convergence in Theorem \ref{thm3} could not be a.s. or be convergence in probability. Indeed, assume that a sequence $\{X^\ve\}$ converges a.s. to a process $X^-(t)\1_{\Omega_-}+X^+(t)\1_{\Omega_+}$, where $\Omega_\pm$ are disjoint measurable sets, $P(\Omega_\pm)=p_\pm.$ It can be shown that $\Omega_\pm\in\cF_{0+},$ so their probabilities are either 0 or 1.   
\end{remk}
\begin{proof}
 Without loss of generality we will assume that
\be\label{eq_A2'}
 \exists c>0\ \forall t\geq 0\  \forall x\in\mbR^d\setminus H: \ \
\sign(x_d) b^d(t,x)\geq c.
\ee
Let us estimate the time spent by $X^\ve$ in the neighborhood of $H.$ By Ito-Tanaka formula, see \cite{RevuzYor},
we have
$$
|X^\ve_d(t)|=\int_0^t\sign(X_d^\ve(s))b_d(s,X^\ve(s))ds +\ve \int_0^t\sign(X_d^\ve(s)) dw_d(s) +L^\ve_d(t)=
$$
$$
=\int_0^t\sign(X_d^\ve(s))b_d(s,X^\ve(s))ds +\ve B^\ve(t) +L^\ve_d(t),
$$
where $B^\ve$ is a new Brownian motion,  $L^\ve_d$ is a non-decreasing, continuous process,  $L^\ve_d(0)=0.$

Therefore, the pair  $(|X^\ve_d|,L^\ve_d)$ is a solution of Skorokhod's reflecting problem for the driving process $\xi_\ve(t)=\int_0^t\sign(X_d^\ve(s))b_d(s,X^\ve(s))ds +\ve B^\ve(t).$ Hence, see for example \cite{PilRSDE},
$$
|X^\ve_d(t)|=\xi_\ve(t)-\min_{s\in[0,t]}\xi_\ve(s).
$$
It follows from  \eqref{eq_A2'} that
$$
|X^\ve_d(t)|\geq (ct+\ve B^\ve(t)) -\min_{s\in[0,t]}(cs+\ve B^\ve(s)).
$$
Denote $\sigma^\ve_{H_\delta}:=\inf\{t\geq 0:\ |X^\ve_d(t)|\geq \delta\}.$

Therefore
\be\label{eq_exit_est1}
P(|X^\ve_d(t)|\geq \delta,\ t\geq 2\delta/c)\to1,\ \ve\to0+,
\ee
\be\label{eq_exit_est2}
P(\sigma^\ve_{H_\delta}>  2\delta/c)\to0,\ \ve\to0+.
\ee
Moreover if   $\delta<1\wedge c/2M$, where $M=\max_{t\in [0,1],\ |x-x^0|\leq 1}|\bar b(t,x)|$, then
\be\label{eq_exit_est3}
P(|\bar X(\sigma^\ve_{H_\delta})-\bar x^0|>  2\delta M/c)\to0,\ \ve\to0+.
\ee

It follows from \eqref{eq_exit_est1}, Lemma \ref{lem_rel_comp}, and assumption \eqref{eq_A2'} that for any limit point $X$ we have with probability 1:
$$
X(t)=x^0+\1_{\Omega_+}\int_0^tb^+(s,X(s))ds+\1_{\Omega_-}\int_0^tb^-(s,X(s))ds,
$$
where $\Omega_+=\{\omega:\ X(t)>0 \mbox{ for all } t>0\},\ \Omega_-=\{\omega:\ X(t)<0\mbox{ for all }   t>0\}$.

Notice that if \eqref{eq_A2'} is true, then
\be\label{eq_284}
P(\Omega_+\cup \Omega_-)=1,\  P(\Omega_-\Delta \{X_d(\sigma_{H_\delta})=-\delta\})=0,\ \mbox{and} \ P(\Omega_+\Delta \{X_d(\sigma_{H_\delta})=\delta\})=0
\ee
for any $\delta>0.$

Let $\delta>0,\ \beta>0$ be sufficiently small fixed numbers and $\alpha$
 be such that for all $ t\in [0,\delta],x\in [x^0-\beta, x^0+\beta]:$
 $$
 0<b_d^+(t,x)-\alpha< b_d^+(0,x_0)< b_d^+(t,x)+\alpha,
 $$
 $$
 0<-b_d^-(t,x)-\alpha< -b_d^-(0,x_0)< -b_d^-(t,x)+\alpha.
 $$
 Define  processes:
$$
X^{\ve,\pm\alpha}(t)=  x^0+\int_0^t(b_d^+(0,x_0)\1_{ X^{\ve,+\alpha}(s)\geq 0}+ b_d^-(0,x_0) \1_{ X^{\ve,+\alpha}(s)< 0}\pm\alpha)ds+\ve w(t), t\geq 0.
$$
By comparison theorem, see \cite{IW}, we have a.s.
$$
X^{\ve,-\alpha}(t)\leq X^d(t)\leq X^{\ve,+\alpha}(t)
$$
for all $t\in [0, \sigma^\ve_{H_\delta}\wedge \inf\{s: |\bar X(s)-\bar x^0|>  \beta\}]$.

The processes $X^{\ve,\pm\alpha}$ are one-dimensional homogeneous diffusions. So, see \cite{GS, IW},
$$
P(X^{\ve,\pm\alpha}_d(\sigma^{\pm\alpha}_{H_\delta})=\delta)=
$$
$$
\frac{(\pm\alpha-b_d^-(0,x_0))^{-1}(1-\exp(2\delta \ve^{-2}(b_d^-(0,x_0)\pm\alpha)))}
{(\pm\alpha-b_d^-(0,x_0))^{-1}(1-\exp(2\delta \ve^{-2}(b_d^+(0,x_0)\pm\alpha)))+(\pm\alpha+b_d^+(0,x_0))^{-1}(1-\exp(2\delta \ve^{-2}(b_d^-(0,x_0)\pm\alpha)))},
$$
where $\sigma^{\pm\alpha}_{H_\delta}:=\inf\{t\geq 0:\ |X^{\ve,\pm\alpha}(t)|\geq \delta\}.$

This and \eqref{eq_284} conclude the proof.
\end{proof}

\begin{thm}
Let $x^0\in H$ and A2$_+$ or  A2$_-$ be satisfied. Then
$$
X^\ve(\cdot\wedge \tau^\pm_H)\to X^\pm(\cdot\wedge \tau^\pm_H), \ \ve\to 0, \ \mbox{a.s.},
$$
where the sign + or - is selected accordingly to the sign in condition A2, $\tau^\pm_H$ is defined in \eqref{eq4.2}.
\end{thm}
\begin{remk}
If condition A2$_\mp$ is satisfied at the point $X^\pm(\tau^\pm_H)$ , then we may define a moment $\tau^2_{\pm}$ similarly to  $\tau^1_{\pm}:=\tau^\pm_H$  and obtain the similar convergence of processes $\{X^\ve\}$ on $[\tau^1_\pm,\tau^2_\mp],$ and so on. Note  that if A2$_+$ is satisfied in $x^0,$ then A2$_+$ cannot be true in $X^+(\tau^1_+).$
\end{remk}
\begin{proof}
Assume that A2$_+$ is satisfied. Without loss of generality  we may assume   that
$$
\exists c>0 \ \forall t\in[0,T]\ \forall x\in\mbR^d_+ :\ \
b^d_+(t,x)\geq c
$$
and
$$
\forall t\in[0,T]\
\forall x\in H:\ \
b^d_-(t,x)\geq 0.
$$

Observe that
$$
X^\ve(t)=x^0+\int_0^t\left(\1_{\{X^\ve(s) >0\}}b^+(s,X^\ve(s))+\1_{\{X^\ve(s)<0\}}b^-(s,X^\ve(s))\right)ds+\ve w(t)=
$$
$$
= x^0+\int_0^t b^+(s,X^\ve(s)) ds+  \int_0^t\1_{\{X^\ve(s) <0\}}\left(b^-(s,X^\ve(s))-b^+(s,X^\ve(s))\right)ds+\ve w(t).
$$
Lemma \ref{lem_rel_comp} and Remark \ref{rem_subseq} yield that  to prove the Theorem it suffices to verify that any limit point $X$ of $\{X_\ve\}$ as $\ve\to 0$ is such that 
\be\label{eq_t3_1}
\int_0^T\1_{\{X_d(s)\leq 0\}} ds=0\ \ \mbox{a.s.}
\ee

  Function $b^-$ is Lipschitz continuous.
So $b_d^-(t,x)\geq Lx_d$, $x\in\mbR^d_-$.

For any $\alpha>0$ set 
$$
X^{\ve,\alpha}_d(t)=\int_0^t \left(\1_{\{X^{\ve,\alpha}_d(s) <0\}} (LX^{\ve,\alpha}_d(s)-\alpha) +\1_{\{X^{\ve,\alpha}_d(s)
 \geq 0\}} c \right)ds +\ve w(t).
$$
By comparison theorem
$X_d^\ve(t)\geq X^{\ve,\alpha}_d(t), t\in [0,T]$ a.s. 

Therefore for any limit point $X$ of $\{X_\ve\}$ 
 is such that 
$$
P(X_d(t)\geq \delta,  \ t\geq 2c/\delta)\geq 
 \limsup_{\ve\to 0}P(X^{\ve}_d(t)\geq \delta,  \ t\geq 2c/\delta)  
$$
$$
\geq \limsup_{\ve\to 0}P(X^{\ve,\alpha}_d(t)\geq \delta,  \ t\geq 2c/\delta)
 \geq P(X^\alpha_d(t)\geq \delta,  \ t\geq 3c/2\delta)=P(X^\alpha_d(t)>0, \ t>0)=\frac{c}{c+\alpha},
$$
 where  $X^\alpha_d$ is a limit of $\{X^{\ve,\alpha}_d\}$ as $\ve\to 0.$
 
 Since $\alpha$  is arbitrary, we have \eqref{eq_t3_1}.
%
%
%
%
%
\end{proof}
\begin{lem}\label{lem1}
Assume that $x^0\in H$, and condition A3 or A4 is satisfied. Let $X^0$ be any (weak) limit point for
$\{X^\ve\}.$
Denote
\be\label{eq_sigma}
\sigma_\delta=\sigma_\delta(X^0):=\inf\{t\geq 0\ : \ |X^0(t)-x^0|\geq \delta\},
\ee
where $\delta>0$ is a parameter from conditions A3, A4.

Then
$$
P(X^0(t)\in H, t\in [0,\sigma_\delta\wedge T])=1.
$$
\end{lem}
\begin{proof}
We prove Lemma if the global condition A4 is satisfied:
$$
 \forall t\geq 0\ \forall x\in H :\ \
  b^d_\pm(t,x)=0.
$$
All other cases are considered similarly.

Assume that $X^{\ve_k}\Rightarrow X^0, k\to\infty.$ By Skorokhod's theorem on a single probability space we may assume that the convergence is a.s.:
\be\label{eq_conv_uniform}
\forall T>0:\ \ \sup_{t\in[0,T]}|X^{\ve_k}(t)- X^0(t)|\to 0, \ k\to\infty.
\ee
Assume that for some $\omega$ and $t_1$ we have $X^0_d(t_1)>0.$ Denote by $t_0\in[0,t_1]$ the last visit of $H$ by $X^0,$
i.e.,
$$
t_0=\sup\{s\in[0,t_1]: \ X^0_d(s)=0\}.
$$
Due to \eqref{eq_conv_uniform} we have
$$
X^0_d(t)=\int_{t_0}^{t_1}b^+_d(s,X^0(s))ds, \ t\in[t_0, t_1].
$$
Since $X^0_d(t_0)=0, X^0_d(t )\geq 0,\  t\in[t_0, t_1], $ by assumption  A4 and Lipschitz property we have
$$
X^0_d(t)\leq L\int_{t_0}^{t_1} X_d^0(s) ds, \ t\in[t_0, t_1].
$$
The application of Gronwall's lemma completes the proof.
\end{proof}

\begin{thm}\label{thm4}
Assume that $x^0=(\bar x^0,0)\in H$ and A3 (or A3$_\pm$) is satisfied, functions $b^\pm$ are continuous in $(t,x)$. Denote
$$
p^\pm(s,\bar x)=\frac{b^\pm_d(s,(\bar x,0))}{b^-_d(s,(\bar x,0))+b^+_d(s,(\bar x,0))}.
$$
Let $\bar X(t)\in\mbR^{d-1}$ be a solution of the ordinary differential equation

$
\bar X^0(t)=\bar x^0+\int_0^t\left(
\bar b^+(s,(\bar X^0(s),0))p^-(s, \bar X^0(s))+\right.
$
\be\label{eq7.1}
\left.
+\bar b^-(s,(\bar X^0(s),0))p^+(s, \bar X^0(s))
\right)ds, \ t\in[0,\sigma_\delta(X^0)\wedge T].
\ee
Then
$$
X^\ve(\cdot\wedge\sigma_\delta(X^0)\wedge T)\mathop{\longrightarrow}\limits^{P} (\bar X^0(\cdot\wedge\sigma_\delta(X^0)\wedge T), 0), \ve\to 0,
$$
\end{thm}
\begin{remk}
Coefficients of the equation \eqref{eq7.1} are continuous in $(s,\bar x)$ and Lipschitz in $\bar x$ until
a solution exits $\delta$-neighborhood of $x^0$ or $t>T.$ Thus there exists a  unique solution to  \eqref{eq7.1}.
\end{remk}
\begin{proof}
We only consider the case
$$
\exists c>0\ \forall t\geq 0\ \forall x\in\mbR^d\setminus H :\ \
\sign(x_d) b^d(t,x)\leq -c.
$$
Let $\{X^{\ve_k}\}$ be any weakly convergent subsequence (see Lemma \ref{lem_rel_comp}). By Skorokhod's theorem on a single probability space we may assume that
\eqref{eq_conv_uniform} holds with probability 1 and also
\be\label{eq_361}
\forall T>0\ \ \sup_{t\in[0,T]} |\int_0^tb_d(s,X^{\ve_k}(s))ds|\to 0, k\to\infty.
\ee
\begin{lem}\label{lem_time}
Let $\omega$ be such that \eqref{eq_conv_uniform} and \eqref{eq_361} are satisfied.
Then
\be\label{eq_time}
\forall t\geq 0:\ \  \lim_{k\to\infty} \int_0^t \1_{X^{\ve_k}(s)\geq 0}ds=\int_0^t\frac{b^-_d(s,X^0(s))}{b^-_d(s,X^0(s))+b^+_d(s,X^0(s))}ds.
\ee
\end{lem}
\begin{remk}
Since all functions under the integral signs in \eqref{eq_time} are bounded,
convergence in Lemma \ref{lem_time} is locally
uniform.
\end{remk}
\begin{proof}
Let  $s_0$ be arbitrary and $\Delta=[t_1,t_2]$ is such that $s_0\in\Delta.$ Then
$$
|\int_\Delta b_d(s,X^{\ve_k}(s))ds| =
\int_\Delta b^-_d(s,X^{\ve_k}(s))\1_{X^{\ve_k}(s)< 0}ds -\int_\Delta b^+_d(s,X^{\ve_k}(s))\1_{X^{\ve_k}(s)\geq 0}ds =
$$
$$
=b^-_d(s_0,X^{\ve_k}(s_0)) \int_\Delta  \1_{X^{\ve_k}(s)< 0}ds -b^+_d(s_0,X^{\ve_k}(s_0))\int_\Delta  \1_{X^{\ve_k}(s)\geq 0}ds
+
$$
$$
+\int_\Delta (b^-_d(s,X^{\ve_k}(s))-b^-_d(s_0,X^{\ve_k}(s_0)))\1_{X^{\ve_k}(s)< 0}ds -\int_\Delta (b^+_d(s,X^{\ve_k}(s))-b^+_d(s_0,X^{\ve_k}(s_0)))\1_{X^{\ve_k}(s)\geq 0}ds .
$$
It follows from the last equality and \eqref{eq_conv_uniform} that
$$
|\int_\Delta b_d(s,X^{\ve_k}(s))ds|  \geq
$$
$$
\geq b^-_d(s_0,X^{0}(s_0)) \int_\Delta  \1_{X^{\ve_k}(s)< 0}ds -b^+_d(s_0,X^{0}(s_0))\int_\Delta  \1_{X^{\ve_k}(s)\geq 0}ds
+o_{\ve_k}(1)|\Delta|+o(|\Delta|)=
$$
\be\label{397}
=|\Delta|\left(
 \frac{b^-_d(s_0,X^{0}(s_0))}{b^-_d(s_0,X^{0}(s_0))+b^+_d(s_0,X^{0}(s_0))}-|\Delta|^{-1}\int_\Delta  \1_{X^{\ve_k}(s)\geq 0}ds
+o_{\ve_k}(1)+o(|\Delta|)/|\Delta|
\right),
\ee
where $o_{\ve_k}(1)$ is independent of $\Delta$, $o_{\ve_k}(1)\to 0, k\to\infty$, and
$o(\Delta)$ is independent of $k.$

If \eqref{eq_time} is not true, then  there exists
a point $s_0$, a subsequence $\{\ve_{k_l}\} $, and a sequence of intervals $\{\Delta_n\},\ \Delta_{n+1}\subset\Delta_n, \ s_0\in\Delta_n, \ n\geq 1, \lim_{n\to\infty}|\Delta_n|=0$ such that
$$
\liminf_{n\to\infty}\liminf_{l\to\infty}
\left|
 |\Delta_n|^{-1}\int_{\Delta_n} \1_{X^{\ve_k}(s)\geq 0}ds-\frac{b^-_d(s_0,X^{0}(s_0))}{b^-_d(s_0,X^{0}(s_0))+b^+_d(s_0,X^{0}(s_0))}
\right|>0.
$$
This contradicts \eqref{397}.
\end{proof}
The following result is well known.
\begin{lem}\label{lem_conv_int}
Assume that $\{f_n\},\ \{l_n\}\subset C([0,T])$ are uniformly convergent sequences of continuous functions
$$
f_n\to f_0 \mbox{ and } l_n\to l_0,\ \ n\to \infty,
$$
and each function $l_n$ is non-decreasing.

Then we have the uniform convergence of the integrals
$$
\sup_{t\in[0,T]}\left
|\int_0^tf_n(s)dl_n(s)-\int_0^tf_0(s)dl_0(s)\right|\to 0,\ n\to\infty.
$$
\end{lem}
The proof of weak convergence $X^\ve\Rightarrow X^0$ follows from Lemmas \ref{lem_time}, \ref{lem_conv_int}, Remark \ref{rem_subseq}, and \eqref{eq_conv_uniform} if we set (recall that $X^{\ve_k}$ are copies)
$f^\pm_n(t):=b^\pm(t,X^{\ve_n}(t)),\ l^\pm_n(t):=
 \int_0^t \1_{X^{\ve_k}(s)\geq 0}ds,\  l^\pm_0(t):=\int_0^t\frac{\pm b^\mp_d(s,X^0(s))}{b^-_d(s,X^0(s))+b^+_d(s,X^0(s))}ds.$ Since $X_0$ is non-random,   weak convergence implies convergence in probability.
\end{proof}

Consider the case A4.

Assume that the  limits exist
$$
c^\pm(t,\bar x):=\lim\limits_{x_d\to 0\pm}\frac{b^\pm_d(t,x)}{x_d}, t\in[0,T], |\bar x-\bar x^0|<\delta.
$$

Since $b^\pm_d(t,x)=0, t\in[0,T], |x-x^0|<\delta,$ and Lipschitz in $x$, functions  $c^\pm_d$ are bounded for
$t\in[0,T], |\bar x-\bar x^0|<\delta.$

Consider the system
\be\label{eq9.1}
\begin{cases}
\bar X(t)=\bar x^0+\int_0^t \left(\bar b^+(s, (\bar X(s),0))\1_{Y(s)\geq 0}+
 \bar b^-(s, (\bar X(s),0))\1_{Y(s)< 0}\right)ds,\\
 Y(t)=\int_0^t \Big( c^+(s, \bar X(s))\1_{Y(s)\geq 0}+  c^{ -}(s, \bar X(s))\1_{Y(s)< 0}\Big)Y(s)ds+w_d(t).
\end{cases}
\ee
\begin{lem}
1) There exists a unique weak solution to \eqref{eq9.1} defined up to the moment $\sigma_\delta(\bar X)\wedge T,$ where $\sigma_\delta(\bar X):=\inf\{t\geq 0\ : \ |\bar X(t)-\bar x^0|\geq \delta\}$.
2) (a) There exists a unique unique strong solution to \eqref{eq9.1} defined up
to the moment $\sigma _{\delta }(\overline{X})\wedge T$, if the functions $%
c^{\pm }(s,\overline{x})=c^{\pm }(s)$ are independent of $\overline{x}$.%
\newline
b) The system \eqref{eq9.1} has a unique maximal solution, if e.g. the
functions $\overline{b}^{\pm }(s,(\overline{x},0)),c^{\pm }(s,\overline{x})$
belong to $C^{3,3}([0,T]\times \mathbb{R}^{d-1}).$

\end{lem}
\begin{proof}
Proof of the weak existence and uniqueness follows from the Girsanov
theorem. Proof in the case 2a) is obvious because all coefficients are
Lipschitz continuous in the spatial variable.

As for the proof of case 2b), see Theorem 3.2 in \cite{LTS}.
\end{proof}

\begin{thm}\label{thm5}
Assume that $x^0\in H$, functions $c^\pm$ are continuous in $(t,x)$, and assumption A4 is satisfied.

Then we have the weak  convergence

$
 \left(\bar X^\ve(\cdot\wedge\sigma_\delta(\bar X^\ve)\wedge T), \ve^{-1} X_d^\ve(\cdot\wedge\sigma_\delta(\bar X^\ve)\wedge T )\right) \to
$
 $$
 \ \ \   \ \ \ \to
\left(\bar X(\cdot\wedge\sigma_\delta(\bar X)\wedge T),  Y(\cdot\wedge\sigma_\delta(\bar X)\wedge T )\right),\ \    \ve\to 0,
 $$
 where $(\bar X,Y)$ is a solution of \eqref{eq9.1}.

In particular $X^\ve(\cdot\wedge\sigma_\delta(\bar X^\ve)\wedge T)\to  \left(\bar X(\cdot\wedge\sigma_\delta(\bar X)\wedge T), 0\right)$ weakly.

 Moreover, if  there exists a strong solution to \eqref{eq9.1}, then   not only weak convergence holds but also  convergence in probability.
\end{thm}
\begin{remk}
Weak uniqueness and strong existence yield uniqueness of the strong solution, see reasoning of \cite{Cherny}.
\end{remk}
\begin{proof}
For simplicity we assume that 
 $$
\forall t\geq 0\ \forall x\in H:\ \
  b^d_\pm(t,x)=0.
$$
Set $Y^\ve(t):=X^\ve_d(t)/\ve.$
Then
$$
\begin{cases}
\bar X^\ve(t)=\bar x^0+\int_0^t \left(\bar b^+(s, (\bar X^\ve(s),Y^\ve(s)))\1_{Y^\ve(s)\geq 0}+
 \bar b^-(s, (\bar X^\ve(s),Y^\ve(s)))\1_{Y^\ve(s)< 0}\right)ds+\ve \bar w(t),\\
 Y^\ve(t)=\int_0^t \Big( \frac{b^+(s, (\bar X^\ve(s),\ve Y^\ve(s)))}{\ve Y^\ve(s)}\1_{Y^\ve(s)\geq 0}+  \frac{b^-(s, (\bar X^\ve(s),\ve Y^\ve(s)))}{\ve Y^\ve(s)}\1_{Y^\ve(s)< 0}\Big)Y^\ve(s)ds+w_d(t).
\end{cases}
$$
\begin{remk}
$\int_0^\infty \1_{Y^\ve(s)=0} ds =0$ a.s.
\end{remk}
It can be readily shown that a family $\{(\bar X^\ve, Y^\ve), \ve\in(0,1)\}$ is weakly relatively compact
in $C([0,\infty); \mbR^d).$ Let  $\{(\bar X^{\ve_n}, Y^{\ve_n})\}$ be a convergent subsequence, where $\lim_{n\to\infty}\ve_n=0$. By theorem on a single probability space, there is a sequence of copies
$(\tilde {\bar X}^{\ve_n}, \tilde Y^{\ve_n}, \tilde   w^{\ve_n} )\ =^d \ (  {\bar X}^{\ve_n},  Y^{\ve_n},    w )$   such that
$$
\begin{cases}
\tilde {\bar X}^{\ve_n}(t)=\bar x^0+\int_0^t \left(\bar b^+(s, (\tilde X^{\ve_n}(s),\tilde Y^{\ve_n}(s)))\1_{\tilde Y^{\ve_n}(s)\geq 0}+
 \bar b^-(s, (\tilde X^{\ve_n}(s),\tilde Y^{\ve_n}(s)))\1_{\tilde Y^{\ve_n}(s)< 0}\right)ds+{\ve_n} \tilde {\bar w}^{\ve_n}(t),\\
 \tilde Y^{\ve_n}(t)=\int_0^t \Big( \frac{b^+(s, (\tilde X^{\ve_n}(s),{\ve_n} \tilde Y^{\ve_n}(s)))}{{\ve_n} \tilde Y^{\ve_n}(s)}\1_{\tilde Y^{\ve_n}(s)\geq 0}+  \frac{b^-(s, (\tilde X^{\ve_n}(s),{\ve_n} \tilde Y^{\ve_n}(s)))}{{\ve_n} \tilde Y^{\ve_n}(s)}\1_{\tilde Y^{\ve_n}(s)< 0}\Big)\tilde Y^{\ve_n}(s)ds+\tilde w^{\ve_n}_d(t).
\end{cases}
$$
and
$$
(\tilde {\bar X}^{\ve_n}, \tilde Y^{\ve_n}, \tilde   w^{\ve_n} ) \to (\tilde {\bar X} , \tilde Y , \tilde   w)
$$
almost surely.

Observe that 
$$
\lim_{n\to\infty}\int_0^t  \Big( \frac{b^+(s, (\tilde X^{\ve_n}(s),{\ve_n} \tilde Y^{\ve_n}(s)))}{{\ve_n} \tilde Y^{\ve_n}(s)}\1_{\tilde Y^{\ve_n}(s)\geq 0}+  \frac{b^-(s, (\tilde X^{\ve_n}(s),{\ve_n} \tilde Y^{\ve_n}(s)))}{{\ve_n} \tilde Y^{\ve_n}(s)}\1_{\tilde Y^{\ve_n}(s)< 0}\Big)\tilde Y^{\ve_n}(s)ds
$$
 is of the form $\int_0^t\xi(s)ds,$ where $\xi(t)$ is independent of $\sigma$-algebra generated by $(\tilde w_d(s)-\tilde w_d(t)), s\geq t.$ So $\int_0^\infty \1_{\tilde Y (s)=0} ds =0$ a.s.

It follows from the Lebesgue dominated convergence theorem that the limit process $(\tilde {\bar X} , \tilde Y)$ is a solution of \eqref{eq9.1} with $\tilde w_d$ in place of $w_d$.
Since  $\{(\bar X^{\ve_n}, Y^{\ve_n})\}$ was arbitrary convergent subsequence, the proof of the Theorem follows from the weak uniqueness of the solution to \eqref{eq9.1}.

If there exists a unique strong solution to \eqref{eq9.1}, consider then a.s. convergent sequence of copies of
$( {\bar X}^{\ve_n},   Y^{\ve_n},     w , {\bar X} ,   Y)$:
$$
(\tilde {\bar X}^{\ve_n}, \tilde Y^{\ve_n},     \tilde w^{\ve_n} , \hat {\bar X}^{\ve_n} , \hat Y^{\ve_n})\to (\tilde {\bar X} , \tilde Y ,     \tilde w  , \hat {\bar X}  , \hat Y ).
$$
It can be seen that the limit processes $(\tilde {\bar X} , \tilde Y )$ and $(\hat {\bar X}  , \hat Y )$ satisfy the same equation with the same Wiener process $\tilde w$. It follows from  uniqueness of the strong solution that $(\tilde {\bar X} , \tilde Y )=(\hat {\bar X}  , \hat Y )$ a.s. So $(\tilde {\bar X}^{\ve_n}, \tilde Y^{\ve_n})- ( \hat {\bar X}^{\ve_n} , \hat Y^{\ve_n})\to 0$ a.s. Therefore $( {\bar X}^{\ve_n},   Y^{\ve_n})- (  {\bar X}  ,   Y )\to 0$  in probability.

\end{proof}
\begin{expl}
A limit of $X^\ve$ may be non-Markov in case A4. Indeed, assume that
$\bar b^\pm(t,x)=\bar b^\pm=\const, \bar b^+\neq \bar b^-, b^\pm_d(t,x)=0,$ and $ x^0\in H.$

It follows from Theorem \ref{thm5} that 
$$
  X^\ve(\cdot)\to (\bar X^0(\cdot), 0), \ve \to 0,
$$
where
$$
 \bar X^0(t):=x^0+ \bar b^+ l^+(t)+ \bar b^- l^-(t),
$$
 $l^\pm(t)$ is the time spent by $w_d(s), s\in[0,t],$ in the positive half-space and negative half-space, respectively.

The process $\bar X^0(t), t\geq 0,$ is not a Markov process.
\end{expl}

\end{document}